\documentclass{amsart}
\usepackage[margin=1in]{geometry}
\usepackage{amssymb, amsmath}
\usepackage{float}

 \usepackage[abs]{overpic}	
\usepackage{enumerate}

\usepackage{amsthm}

\usepackage{amssymb, amsmath}
\usepackage{float}

\usepackage[colorlinks = true,
            linkcolor = red,
            urlcolor  = red,
            citecolor = red,
            anchorcolor = red]{hyperref}
\usepackage{amsrefs}

\newtheorem{theorem}{Theorem}[section]

\newtheorem{lemma}[theorem]{Lemma}

\begin{document}

   \title{Surface-link families with arbitrarily large triple point number}
   
   \author{Nicholas Cazet}

   \begin{abstract}
   
   Analogous to a classical link diagram, a surface-link can be generically projected to 3-space and given crossing information to create a broken sheet diagram.  The triple point number of a surface-link is the minimal number of triple points among all  broken sheet diagrams that lift to that surface-link. This paper generalizes a family of Oshiro to show that there are non-split surface-links of arbitrarily many trivial components whose triple point number can be made arbitrarily large. 
   
 \end{abstract}

\maketitle

\section{Introduction}

A {\it surface-link} is a smoothly embedded closed surface  in $\mathbb{R}^4$. A {\it 2-knot} is a surface-knot diffeomorphic to the 2-sphere, and a {\it surface-knot} will mean a connected surface-knot. Two surface-links are {\it equivalent} if they are related by an ambient isotopy in the smooth category. For an orthogonal projection $p: \mathbb{R}^4\to\mathbb{R}^3$, a surface-knot $F$ can be perturbed slightly so that $p(F)$ is a generic surface. Each point of $p(F)$ has a neighborhood in 3-space diffeomorphic  to $\mathbb{R}^3$ such that the image of the  generic surface under the diffeomorphism looks like 1, 2, or 3 coordinate planes or the cone on a figure 8 (Whitney umbrella).
These points are called regular points, double points,  triple points, and  branch points. Triple points and branch points are isolated while double points are not and lie on curves called {\it double point curves}. The union of non-regular points is the {\it singular set} of the generic surface \cite{CKS}, \cite{kamada2017surface}.

\begin{figure}[ht]
\includegraphics[scale=.25]{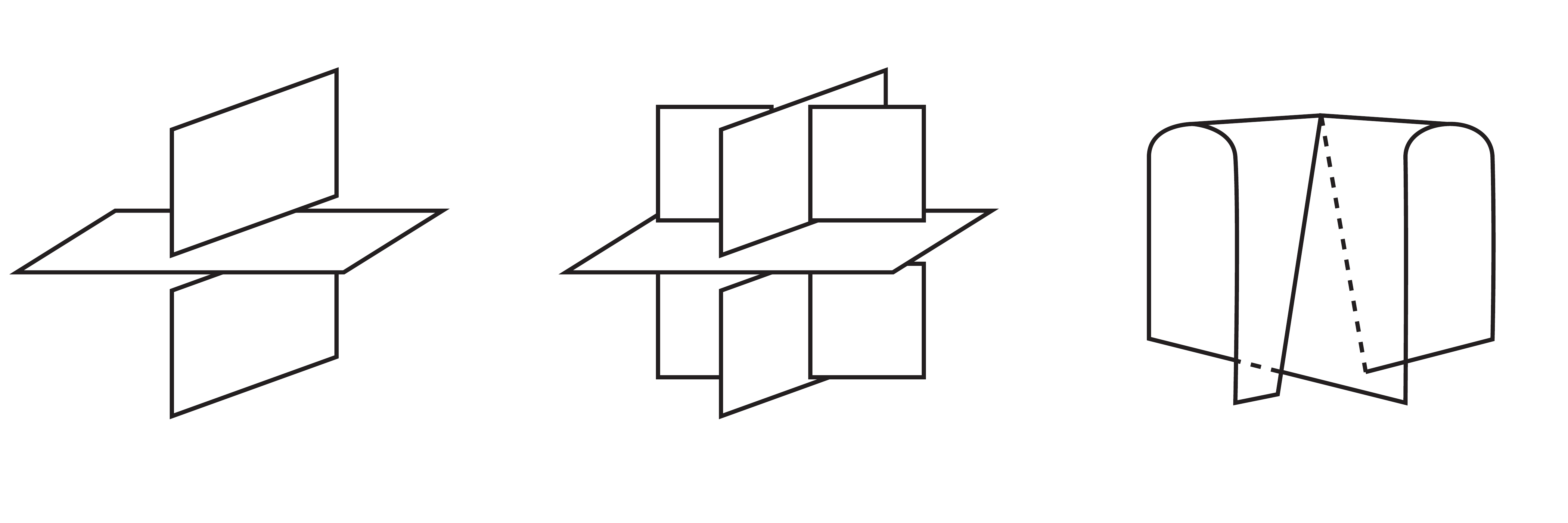}
\caption{Local images of a double point, triple point, and branch point.}
\label{fig:1}
\end{figure}

A {\it broken sheet diagram} of a surface-link $F$ is a generic projection $p(F)$ with  consistently broken sheets along double point curves,  see Figure \ref{fig:1} and \cite{cartersaito}. The sheet that lifts below the other, with respect to a height function determined by the direction of the orthogonal projection $p$, is locally broken at the singular set. All surface-links admit a broken sheet diagram, and all broken sheet diagrams lift to surface-knots in 4-space. Although, not all compact, generic surfaces in 3-space can be given a broken sheet structure  \cite{Carter1998}.

The minimal number of triple points among all generic projections of a surface-link $F$ is called the {\it triple point number} of $F$ and is denoted $t(F)$. The triple point number has an analogy to the crossing number of a classical knot. Although it is unknown if the crossing number of classical knots is additive under connected sum, Satoh showed that the connected sum of the $n$-twist-spin of a 2-bridge knot with the non-orientable trivial surface-knot of genus 3 and normal Euler number $\pm2$  produces a surface-knot whose triple point number is zero while it is known that these twist-spins have positive triple point number \cite{Satohconnect}. 

Satoh and Shima proved many foundational results on the triple point number in the early 2000's:

\begin{theorem}[Satoh '00 \cite{satoh2}]
 No surface-knot has a triple point number of 1. 
\end{theorem}

\begin{theorem}[Satoh '01 \cite{satoh3}]
For any positive $n$, there exists a surface-knot whose triple point number is $2n$.

\end{theorem}

\begin{theorem}[Satoh-Shima '03 \cite{satoh1}]
The 2-twist-spun trefoil has a triple point number of 4.
\end{theorem}

\begin{theorem}[Satoh-Shima '05 \cite{satoh6}]
The 3-twist-spun trefoil has a triple point number of 6.
\end{theorem}

\begin{theorem}[Satoh '05 \cite{satoh7}]
No 2-knot has a triple point number of 2 or 3.
\end{theorem}

Hatakenaka generalized the method used to compute the triple point number of the 2-twist-spun trefoil to show that the triple point number of the 2-twist-spun figure eight knot is between 6 and 8 and the triple point number of the 2-twist-spun (2,5)-torus knot is between 6 and 12  \cite{hat}. Satoh then returned to the problem to calculate these surface-knots' exact triple point number:

\begin{theorem}[Satoh '16 \cite{satohcocycle}]
The 2-twist-spun figure-eight knot and the 2-twist-spun (2,5)-torus knot have a triple point number of 8.
\end{theorem}

It is known that no surface-knot of genus one has a triple point number of 2 \cite{ky}. Currently, the only examples of surface-knots with triple point number of two are non-orientable surface-links. The only calculated triple point numbers have been even, although there are too few examples to suggest that the triple point number must always be even.

There are few infinite families of surface-knots with calculated or bounded triple point numbers. Kamada provided the first result of the kind:

\begin{theorem}[Kamada '93 \cite{kamadatriple}] For any positive integer $n$, there exists some 2-knot $S$ such that $t(S)>n$.\label{thm2}
\end{theorem}

\noindent Kamada's algebraic proof allows for the addition of trivial handles to generalize the result to connected, orientable surface-knots of any genus.

In 2001, Satoh showed that  non-split 2-component surface-links whose components are trivial, non-orientable, and of arbitrary genus can achieve arbitrarily large triple point number \cite{satoh3}. His lower bound calculation relies on each component being non-orientable, $P^2$-irreducible, and having nonzero normal Euler number. In 2009, Kamada and Oshiro showed a similar result using symmetric quandles  \cite{sym}.  In 2010, Oshiro further explored their method to show that there are non-split 2-component surface-links with both components  trivial and non-orientable whose triple point number can be made arbitrarily large regardless of normal Euler number \cite{oshiro}.

Theorem \ref{thm} generalizes Oshiro's family and calculation by adding trivial components and extending the original quandle coloring:

\begin{theorem}

For any non-negative integers  $k$ and  $m$, there exists a non-split $k+m+1$-component surface-link $F = \cup_{i=1}^k F_i \ \cup \ \cup_{i=1}^m F'_i \ \cup \ G$ such that 

\begin{itemize}

\item[(i)] $F_i$ is trivial and orientable of arbitrary genus $g_i$, 
\item[(ii)] $F'_i$ is trivial and non-orientable of arbitrary even genus $g'_i$,
\item[(iii)]$G$ is trivial and orientable of genus $m+k$, 
\item[(iv)] $F - G$ is a trivial surface-link,
\item[(v)] $t(F)=\sum_{i=1}^m g'_i$.

\end{itemize}

\label{thm}
\end{theorem}

 Section \ref{sec:3} provides background information on the weight of symmetric quandle 3-cocycles, Section \ref{sec:4} describes the induced broken sheet diagram of a motion picture,   and Section \ref{sec:7} proves Theorem \ref{thm}.

\section{The Weight of a Symmetric Quandle 3-cocycle}
\label{sec:3}

A {\it quandle}  is a set $X$ with a binary operation $(x,y)\mapsto x^y$ such that 

\begin{enumerate}

\item[(i)] for any $x\in X$, it holds that $x^x=x$,
\item[(ii)] for any $x,y\in X$, there exists a unique $z\in X$ such that $z^y=x$, and 
\item[(iii)] for any $x,y,z\in X$, it holds that $(x^y)^z=(x^z)^{(y^z)}$.

\end{enumerate}

\noindent For a quandle $X$, a {\it good involution $\rho$ of $X$} is an involution such that 

\begin{enumerate}
\item[(i)] for any $x,y\in X$, $\rho(x^y)=\rho(x)^y$, and 
\item[(ii)] for any $x,y\in X$, $x^{\rho(y)}=x^{y^{-1}}$.
\end{enumerate}

\noindent A quandle paired with a good involution is called a {\it symmetric quandle}.

 Let $(X,\rho)$ be a symmetric quandle and $A$ an abelian group. A homomorphism $\phi: \mathbb{Z}(X^3)\to A$ is a {\it symmetric quandle 3-cocyle of $(X,\rho)$} if the following conditions are satisfied: 

\begin{enumerate}
\item[(i)] For any $(a,b,c,d)\in X^4,$ \[ \phi(a,c,d)-\phi(a^b,c,d)-\phi(a,b,d)+\phi(a^c,b^c,d)+\phi(a,b,c)-\phi(a^d,b^d,c^d)=0,
\] 
\item[(ii)] for any $(a,b)\in X^2$, $\phi(a,a,b)=0$ and $\phi(a,b,b)=0$, and
\item[(iii)] for any $(a,b,c)\in X^3$, \[ \phi(a,b,c)+\phi(\rho(a),b,c)=0,\] \[ \phi(a,b,c)+\phi(a^b,\rho(b),c)=0, \] \[\text{and} \quad  \phi(a,b,c)+\phi(a^c,b^c,\rho(c))=0.\]
\end{enumerate}

\noindent For any symmetric quandle $(X,\rho)$, there is an associated chain and cochain complex. Symmetric quandle 3-cocycles are cocycles of this cochain complex and represent cohomology classes of $H^3_{Q,\rho}(X;A)$,  see  \cite{carter2009symmetric}, \cite{kamada2017surface}, \cite{sym}.

 \begin{figure}[h]
\centering
\begin{overpic}[unit=.5mm,scale=.65]{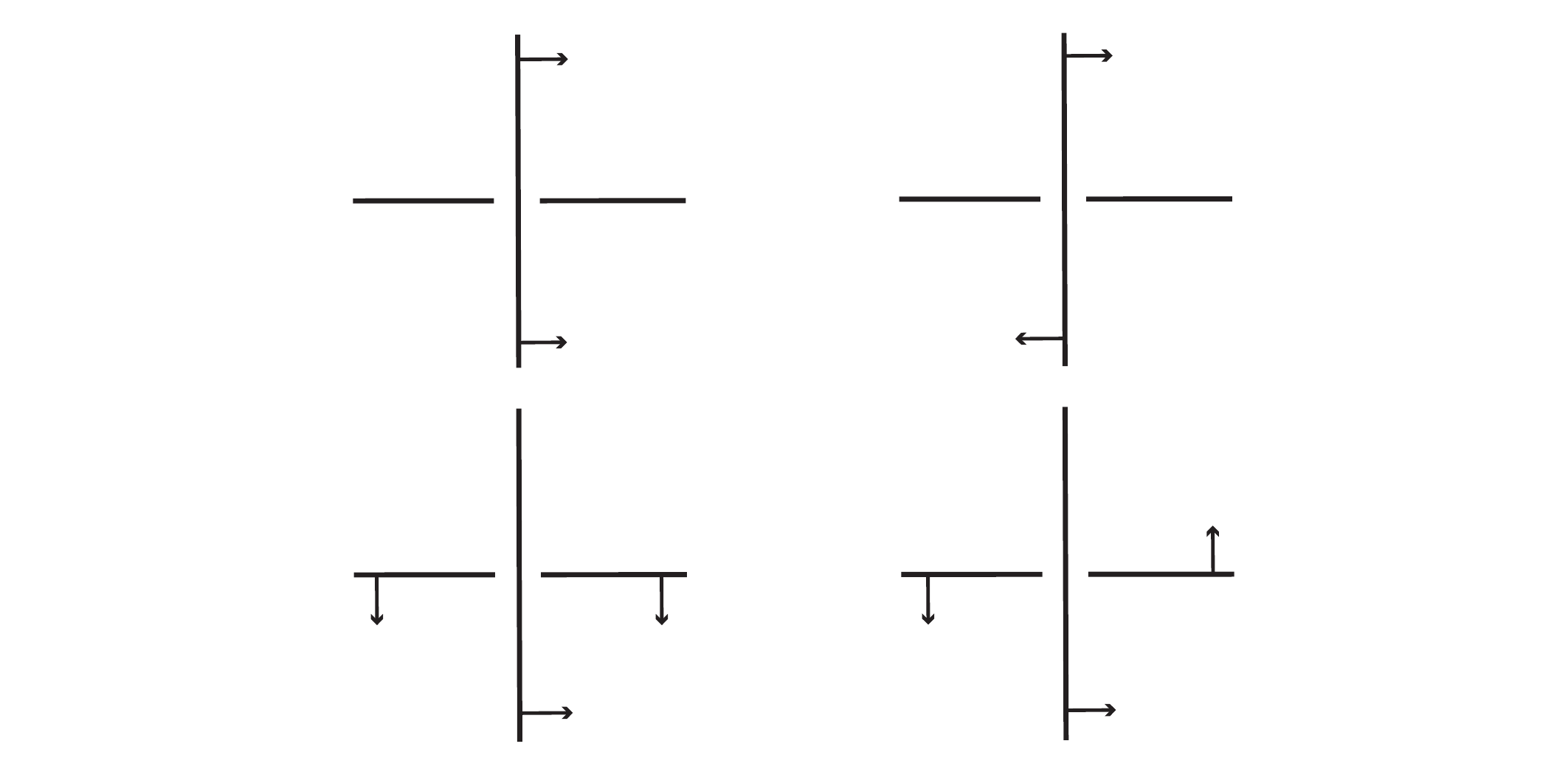}\put(82,115){$x_1$}
\put(82,86){$x_2$}

\put(180,115){$x_1$}
\put(180,86){$x_2$}

\put(115,80){${x_1=x_2}$}\put(213,80){ ${x_1=\rho(x_2)} $}

\put(74,39){$x_1$}\put(108,39){$x_2$}\put(82,17){$x_3$}

\put(170,39){$x_1$}\put(204,39){$x_2$}\put(180,17){$x_3$}

\put(115,11){${x_1^{x_3}=x_2}$}\put(213,11){${x_1^{x_3}=\rho(x_2)}$}

\end{overpic}

\caption{Coloring conditions of a link diagram.}
\label{fig:colorlink}
\end{figure}
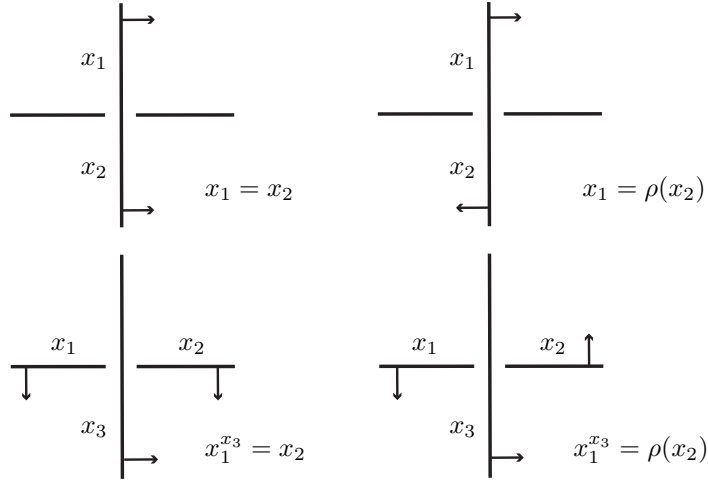

Let $L$ be a classical link diagram. Divide over-arcs at each crossing to produce the {\it semi-arcs} of $L$. For a symmetric quandle $(X,\rho)$, an assignment of a normal orientation and $X$ elements to each semi-arc satisfies the {\it coloring conditions} if the following hold:

\begin{enumerate}
\item[(i)] Suppose that the two adjacent semi-arcs coming from an over-arc at a crossing of $L$ are labeled with $X$ elements $x_1$ and $x_2$. If the normal orientations are coherent then $x_1=x_2$, otherwise $x_1=\rho(x_2)$. See the top row of Figure \ref{fig:colorlink}.

\item[(ii)] Suppose that the two under-arcs at a crossing are labeled with $X$ elements $x_1$ and $x_2$, and that one of the semi-arcs coming from the over-arc is  labeled $x_3$ with a normal orientation pointing toward the under-arc labeled with $x_2$.  If the normal orientations of the under-arcs are coherent, then $x_1^{x_3}=x_2$, otherwise $x_1^{x_3}=\rho(x_2)$. See the bottom row of Figure \ref{fig:colorlink}.

\end{enumerate}

An {\it $(X,\rho)$-coloring} of $L$ is the equivalence class of an assignment of normal orientations and elements of $X$ to the semi-arcs of $L$ satisfying the coloring conditions. The equivalence relation is generated by {\it basic inversions}. Such an inversion reverses the normal orientation of a semi-arc and changes the assigned element $x$ to $\rho(x)$.

Let $D$ be a broken sheet diagram. Divide over-sheets along the double point curves and call the result {\it semi-sheets} of $D$. Note that every semi-sheet is orientable even if $F$ is non-orientable. For a symmetric quandle $(X,\rho)$, an assignment of a normal orientation and an element of $X$ to each semi-sheet satisfies the {\it coloring conditions} if the following hold:

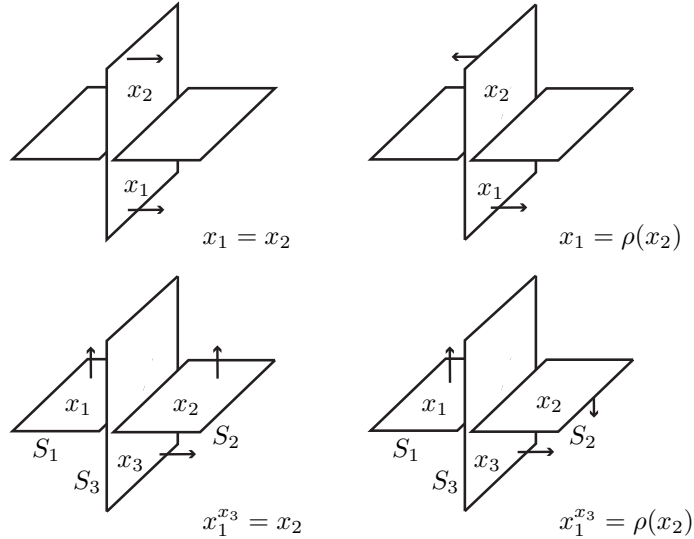
\begin{figure}[h] 
\centering
\begin{overpic}[unit=.435mm,scale=.7]{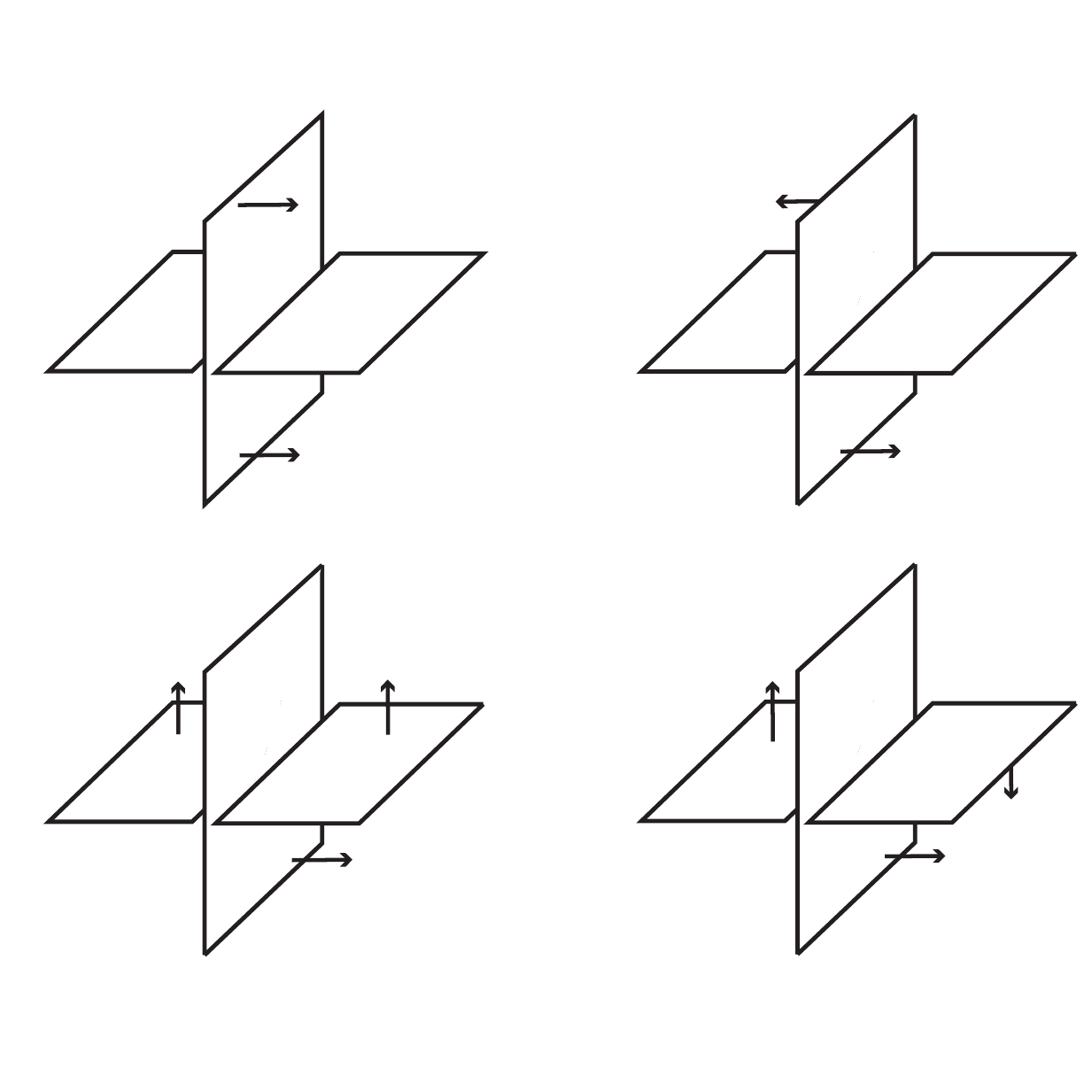}  \put(43,123){$x_1$}  \put(44,152){$x_2$} \put(151,122){$x_1$}  \put(153,152){$x_2$} 
 \put(27.5,32){$S_3$}  \put(150,39){$x_3$} 
 \put(15,42){$S_1$}  \put(134,56){$x_1$} 
  \put(70,46){$S_2$}  \put(169,57){$x_2$} 
  \put(67,107){${x_1=x_2}$}  \put(176,107){${x_1=\rho(x_2)}$}
  
  \put(67,20){$x_1^{x_3}=x_2$}  \put(176,20){$x_1^{x_3}=\rho(x_2)$}
  
  \put(25,56){$x_1$}  \put(125,42){$S_1$}
  
    \put(58,56){$x_2$}  \put(179,46){$S_2$}
    
        \put(41,39){$x_3$}  \put(137,32){$S_3$}
\end{overpic}
\caption{Coloring conditions of a broken sheet diagram.}
\label{fig:sheet}
\end{figure}

\begin{enumerate}
\item[(i)] Suppose that two adjacent semi-sheets coming from an over-sheet of $D$ about a double point curve are labeled by $x_1$ and $x_2$. If the normal orientations are coherent then $x_1=x_2$, otherwise $x_1=\rho(x_2)$. See the top row of Figure \ref{fig:sheet}.

\item[(ii)] Suppose that two adjacent semi-sheets $S_1$ and $S_2$ coming from under-sheets about a double point curve are labeled by $x_1$ and $x_2$, and that one of the two semi-sheets coming from an over-sheet of $D$, say $S_3$, is labeled by $x_3$. Assume that the normal orientation of $S_3$ points from $S_1$ to $S_2$. If the normal orientations of $S_1$ and $S_2$ are coherent, then $x_1^{x_3}=x_2$, otherwise $x_1^{x_3}=\rho(x_2)$. See the bottom row of Figure \ref{fig:sheet}.

\end{enumerate}

An {\it $(X,\rho)$-coloring} of $D$ is the equivalence class of an assignment of normal orientations and elements of $X$ to the semi-sheets of $D$ satisfying the coloring conditions. The equivalence relation is generated by {\it basic inversions}. Such an inversion reverses the normal orientations of a semi-sheet and changes the assigned element $x$  to $\rho(x)$ \cite{sym}, \cite{oshiro}.

 Let $C$ be an $(X,\rho)$-coloring of a broken sheet diagram $D$. For a triple point $\tau$ of $D$, choose one of the eight 3-dimensional complementary regions around the triple point and call the region  {\it specified}. There are 12 semi-sheets around a triple point. Let $S_B$, $S_M$, and $S_T$ be the three of them that face the specified region, where $S_B$, $S_M$, and $S_T$ are semi-sheets of the bottom, middle, and top sheet respectively. Let $n_B$, $n_M$, and $n_T$ be the normal orientations of $S_B$, $S_M$, and $S_T$. Through basic inversions, it is assumed that each normal orientation points away from the specified region. Let $x,y,$ and $z$ be the elements of $X$ assigned to the semi-sheets $S_B$, $S_M$, and $S_T$ whose normal orientations $n_B$, $n_M$, and $n_T$ point away from the specified region. The {\it color} of the triple point $\tau$ is the triple $C_\tau=(x,y,z).$ 
 
\begin{figure}[h]
\centering
\begin{overpic}[unit=.35mm,scale=.35]{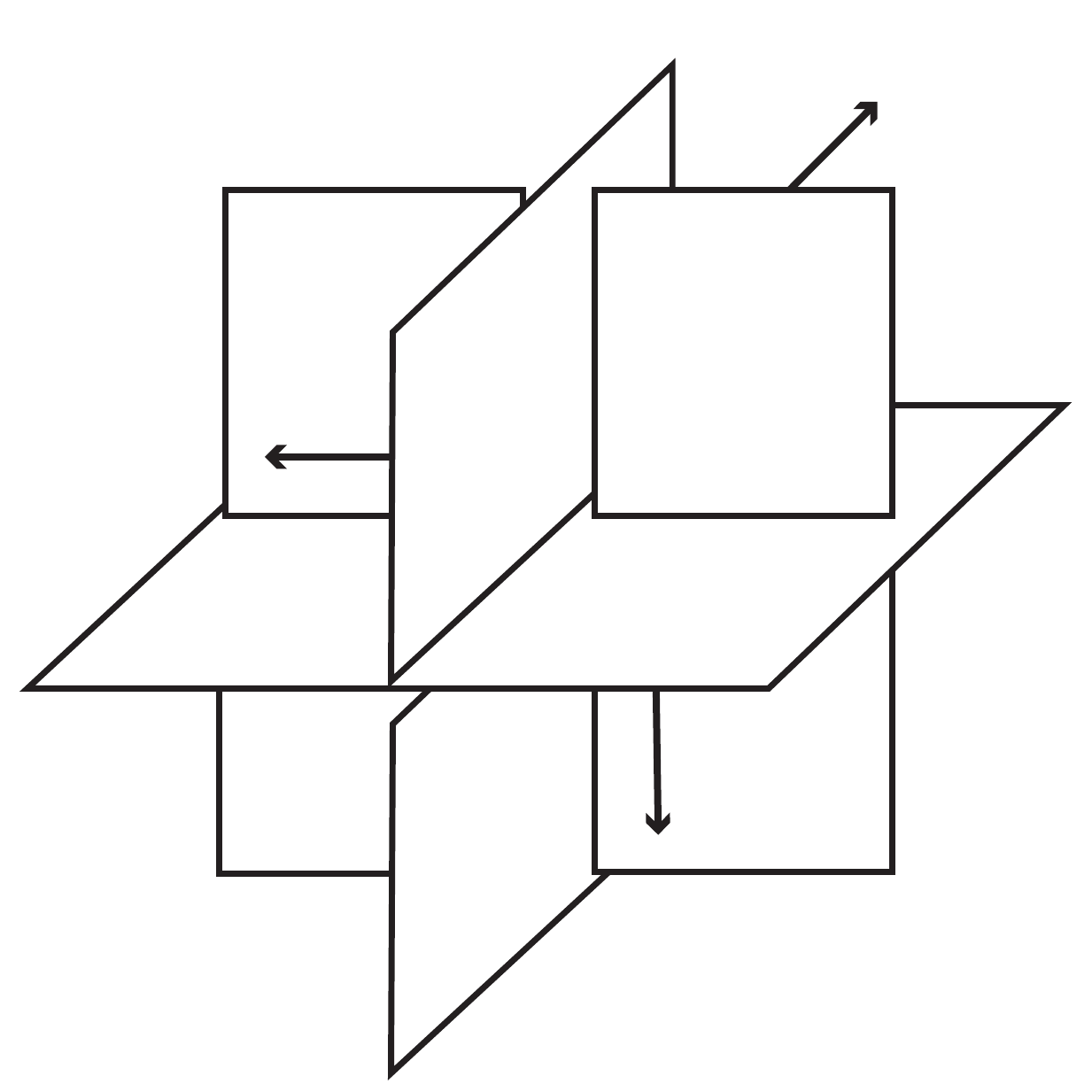}\put(82,82.5){$x$}  \put(56,77){$y$}  
\put(74,55){$z$}
\end{overpic}
\caption{A positive colored triple point with $C_\tau=(x,y,z)$.}
\label{fig:spec}
\end{figure}

 Let $\phi:\mathbb{Z}(X^3)\to A$ be a symmetric quandle 3-cocycle of $(X,\rho)$. The ${\it \phi-weight}$ of the triple point is defined by $\epsilon \phi(x,y,z)$ such that $\epsilon$ is +1 (or -1) if the triple of the normal orientations $(n_T,n_M,n_B)$ is (or is not) coherent with the orientation of $\mathbb{R}^3$ at the triple point. The triple point of Figure \ref{fig:spec} is positive. The ${\it \phi-weight}$ of a diagram $D$ with respect to a symmetric quandle coloring $C$ is \[ \phi(D,C)=\sum_\tau (\phi-\text{weight \ of \ $\tau$})\in A,\] where $\tau$ runs over all triple points of $D$. The value $\phi(D,C)$ is an invariant of an $(X,\rho)$-colored surface-link $(F,C)$ \cite{sym}, \cite{oshiro}. Denote $\phi(D,C)$ by $\phi(F,C)$.

\section{Induced Broken Sheet Diagram of a Motion Picture}
\label{sec:4}

Given a surface-knot $F\subset \mathbb{R}^4$ and a vector ${\bf v}\in \mathbb{R}^4$, perturb $F$ such that the orthogonal projection of $\mathbb{R}^4$ onto $\mathbb{R}$ in the direction of ${\bf v}$ is a Morse function. For any $t\in \mathbb{R}$, let $\mathbb{R}_t^3$ denote the affine hyperplane orthogonal to ${\bf v}$ that contains the point $t {\bf v}$.  Morse theory allows for the assumption that all but finitely many of the non-empty cross-sections $F_t=\mathbb{R}_t^3 \cap F$ are classical links. The decomposition $\{F_t\}_{t\in\mathbb{R}}$ is called a {\it motion picture} of $F$. It may also be assumed that the exceptional cross-sections contain minimal points, maximal points, and/or immersed links with double points representing saddles \cite{CKS}, \cite{kamada2017surface}.

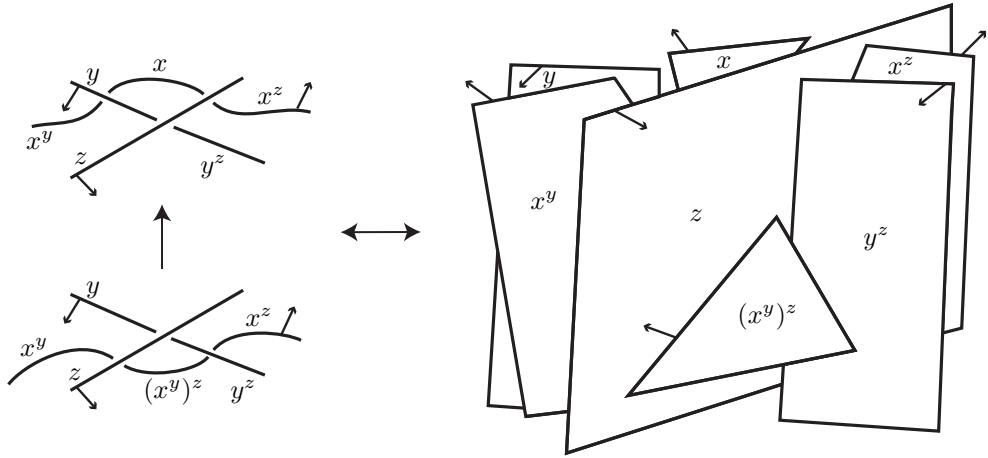
\begin{figure}[h]
 \centering
\begin{overpic}[unit=.399mm,scale=.6]{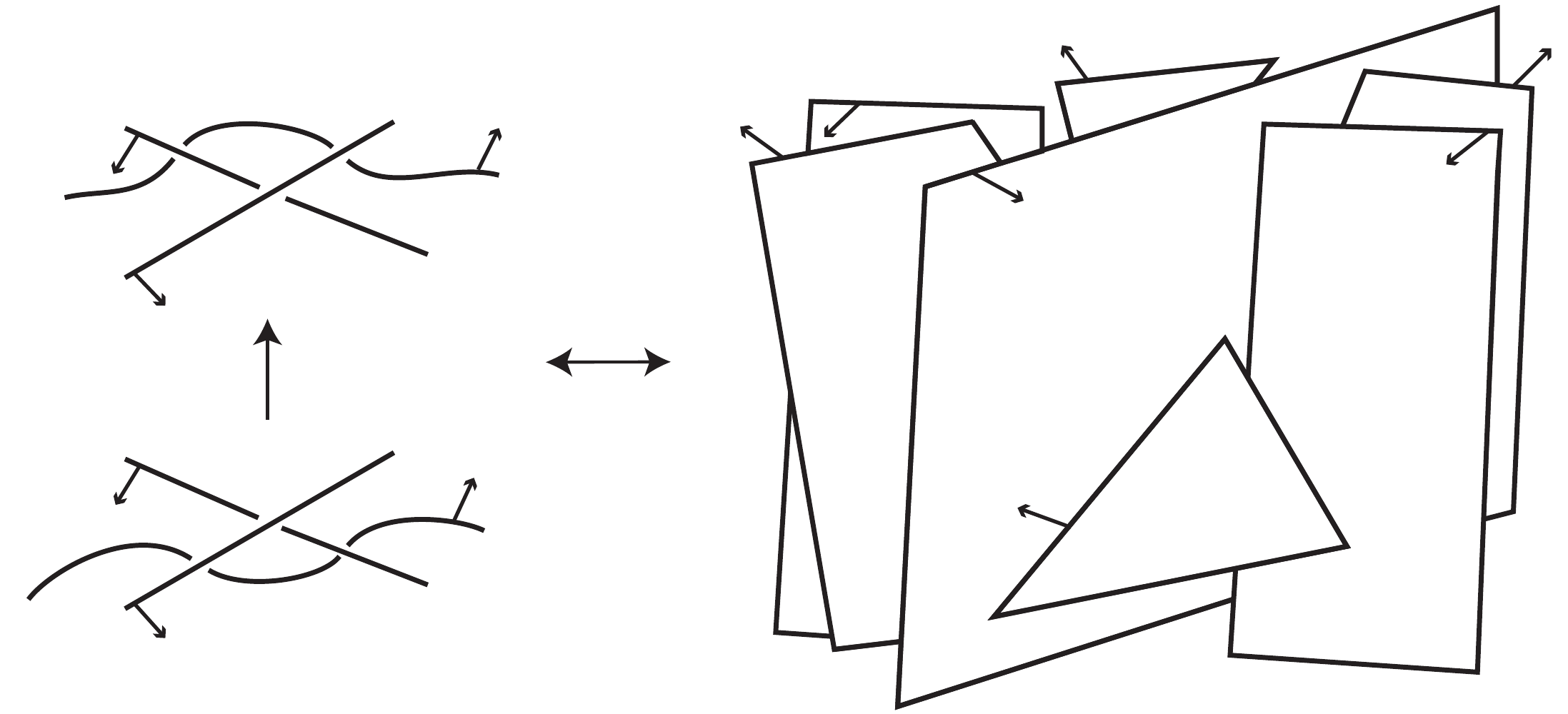}\put(26,26){$z$}\put(10,34){$x^y$}
\put(32,55){$y$}\put(32,126){$y$}\put(12,103){$x^y$}\put(28,98){$z$}\put(89,118){$x^z$} 
\put(54,129){$x$}\put(242,130){$x$}\put(70,95){$y^z$}\put(80,20){$y^z$}\put(51,20){$(x^y)^z$}

\put(184,125){$y$}\put(180,83){$x^y$}\put(233,79){$z$}\put(249,45){$(x^y)^z$}\put(291,71){$y^z$}\put(299,128){$x^z$}\put(85,44){$x^z$}
\end{overpic}

\caption{Relationship between Reidemeister III moves and triple points.}
\label{fig:r3}

 \end{figure}
 
There is a product structure between Morse critical points implying that only finitely many cross-sections are needed to decompose, or construct, $F$. Although, a sole cross-section of a product region does not uniquely determine its knotting, ambient isotopy class relative boundary, see \cite{CKS}. Project the cross-sections $\{F_t\}_{t\in\mathbb{R}}$  onto a plane to get an ordered family of planar diagrams containing classical link diagrams, minimal points, maximal points, and link diagrams with transverse double points. These planar diagrams are {\it stills} of the motions picture. The collection of all stills will also be referred to as a motion picture. The double points of the immersed link diagrams can be replaced with bands to give information about the double points' smoothings  in the stills immediately before and after.

 \begin{figure}[ht]

\includegraphics[scale=.5]{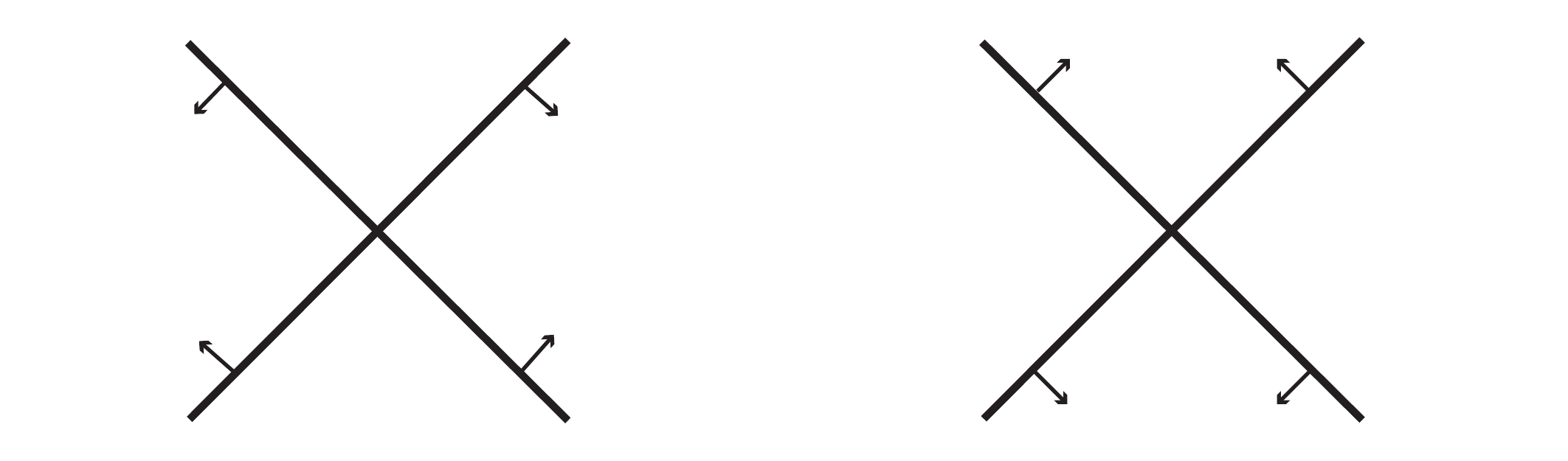}

\caption{Orientation at a saddle point.}
\label{fig:markedv}
\end{figure}

The product structure between critical points also implies that cross-sections between consecutive critical points represent the same link. Therefore, there is a sequence of Reidemeister moves and planar isotopies between the stills of a motion picture that exists between consecutive critical points. 

A motion picture induces a broken sheet diagram. Associate the time parameter of a Reidemeister move with the height of a local broken sheet diagram. A translation of each Reidemeister move to a broken sheet diagram is done in  \cite{kamadakim}. A Reidemeister III move gives a triple point diagram seen in Figure \ref{fig:r3}, a Reidemeister I move corresponds to a branch point, and a Reidemeister II move corresponds to a maximum or minimum of a double point curve, Figures 5 \& 6 of \cite{kamadakim}.

\noindent The triple points of the induced broken sheet diagram are in corresponds with the Reidemeister III moves between stills.

For a symmetric quandle $(X,\rho)$, an $(X,\rho)$-coloring of an immersed link diagram with transverse double points is an assignment of $X$ elements to each arc such that the symmetric coloring conditions are satisfied at each crossing, each of the four arcs at a double point are given the same color, and the normal orientations of the arcs at a double point satisfy Figure \ref{fig:markedv}. Geometric justification of the coloring constraints at a double point, as well as an example of replacing a double point with a band, is shown in Figure \ref{fig:saddle}.

 \begin{figure}[h]
\centering
\includegraphics[scale=.65]{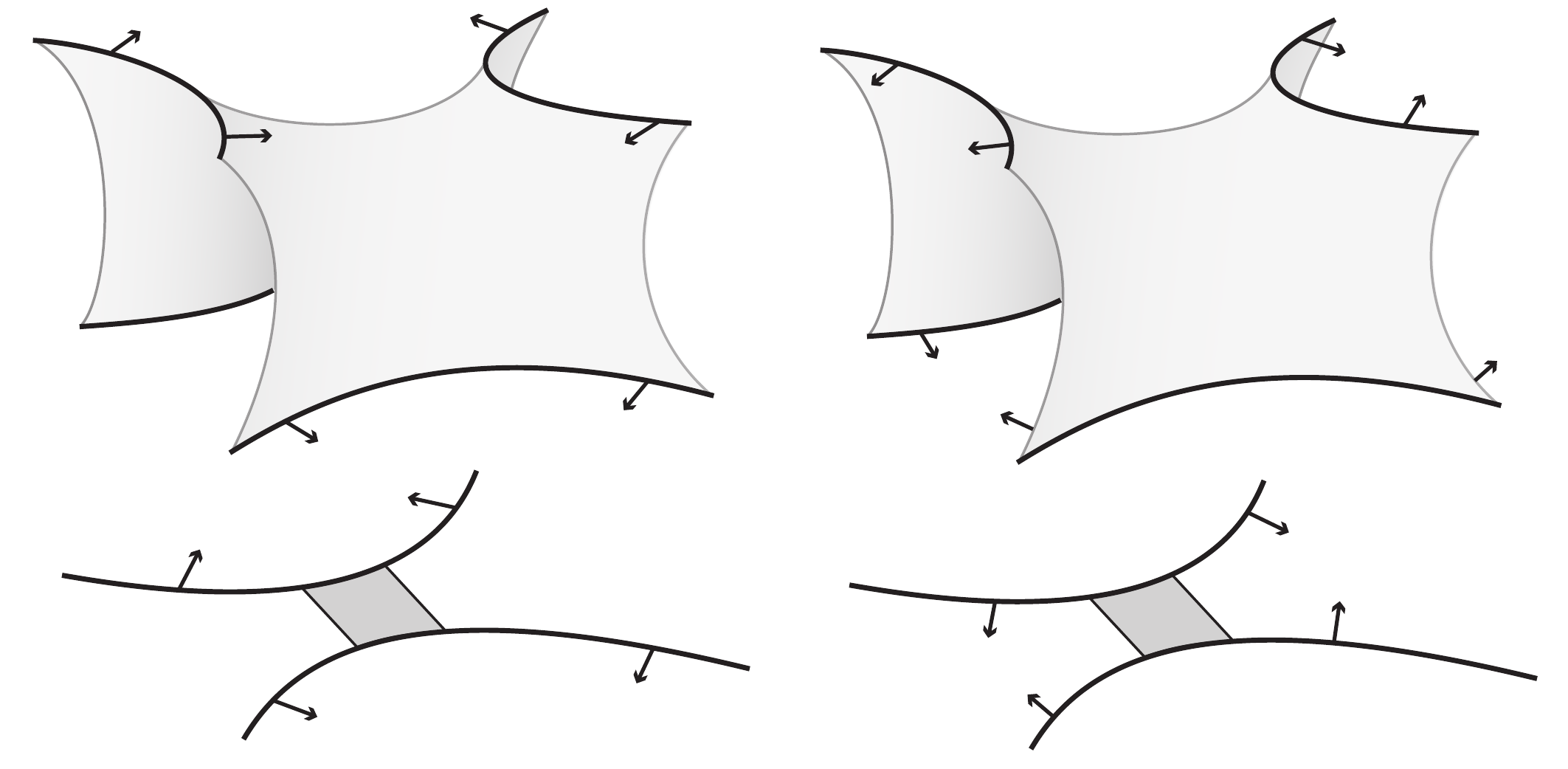}

\caption{Induced saddle sheet.}
\label{fig:saddle}
\end{figure}

An $(X,\rho)$-coloring of a motion picture is a {\it consistent} $(X,\rho)$-coloring of each still. Consistent means that stills separated by a Reidemeister move, or planar isotopy, have colorings consistent with the unique coloring extension of the move. An $(X,\rho)$-coloring of a motion picture gives an $(X,\rho)$-coloring of the induced broken sheet diagram. Give each induced sheet the same color as any arc that traces the sheet in the motion picture. With the addition of the appropriately colored saddle sheets,  an $(X,\rho)$-coloring of the entire broken sheet diagram is achieved.

\section{Proof of Theorem \ref{thm}}
\label{sec:7}

For non-negative integers $s$ and $t$, let $A_{s,t}$ denote the direct sum of $s$ copies of $\mathbb{Z}_2$ and $t$ copies of $\mathbb{Z}$, $A_{s,t}=(\mathbb{Z}_2)^s\oplus(\mathbb{Z})^t$. Every element of $A_{s,t}$ is of the form $(\alpha_1 \oplus \cdots \oplus \alpha_s) \oplus ( \beta_1\oplus \cdots \oplus \beta_t)$, where $\alpha_i$ is an entry of the $i$th copy of $\mathbb{Z}_2$ and $\beta_j$ is an entry of the $j$th copy of $\mathbb{Z}$. Let $p_i$ and $q_j$ be the elements of $A_{s,t}$ whose entries are all zeros except $\alpha_i=1$ and $\beta_j=1$.

Consider a broken sheet diagram realizing the triple point number of the surface-knot it represents. If a symmetric quandle 3-cocycle has $\mathbb{Z}$ coefficients and only takes the values $1, -1$ or 0, then the absolute value of the cocycle's weight cannot be greater than the number of triple points in the diagram. Since the weight of a symmetric quandle 3-cocycle is an invariant, the triple point number bounds the weight of the cocycle. This is the principle of the following lemma.

\begin{lemma}[Oshiro '10 \cite{oshiro}]

Let $(X,\rho)$ be a symmetric quandle, and let $\phi: \mathbb{Z}(X^3)\to A_{s,t}$ be a 3-cocycle of $(X,\rho)$ such that for any generator $(a,b,c)\in X^3$ of $\mathbb{Z}(X^3)$ it holds that \[ \phi(a,b,c)\in\{0,p_i,\pm q_j \}.\]

If the invariant $\phi(F,C)$ of a surface-knot $F$ with an $(X,\rho)$-coloring $C$ is equal to $(\alpha_1 \oplus \cdots \oplus \alpha_s) \oplus ( \beta_1\oplus \cdots \oplus \beta_t)$, then we have $t(F)\geq \sum_{i=1}^s\alpha_i + \sum_{i=1}^t|\beta_j|$ where the sum is taken in $\mathbb{Z}$ by regarding $\alpha_k=0$ or 1 as an element of $\mathbb{Z}$.

 \label{lem} \end{lemma}

Let $P_3$ be the quandle whose multiplication table is shown in Table \ref{tab:2}. The involution $\rho: P_3\to P_3$ defined by $\rho(0)=0$ and $\rho(1)=2$ is a good involution of $P_3$ \cite{oshiro}.

\begin{table}[h]
\centering
\begin{tabular}{c|c c c  }

$P_3$ & 0&1&2 \\\hline 
 0 & 0 &0&0\\ 
1 & 2&1&1\\
2 & 1 &2&2\\

\end{tabular}
\vspace{4mm}
\caption{Multiplication table of $P_3$.}
\label{tab:2}
\end{table}

Define a map $\theta:P_3^3\to\mathbb{Z}_2\oplus\mathbb{Z}$ such that 
\[   \theta(a,b,c)=\left\{
\begin{array}{ll}
      1\oplus 0 & \quad (a,b,c)\in\{(0,1,0),(0,2,0)\}, \\
      0\oplus 1 & \quad (a,b,c)\in\{(1,0,2),(2,0,1)\},  \\
     0\oplus -1 & \quad (a,b,c)\in\{(1,0,1),(2,0,2)\},  \\
      0\oplus0 & \quad \text{otherwise.} \\
\end{array} 
\right. \]
The linear extension $\theta:\mathbb{Z}(P_3^3)\to\mathbb{Z}_2\oplus\mathbb{Z}$ is a symmetric quandle 3-cocycle of $(P_3,\rho)$ \cite{oshiro}. This cocycle satisfies the assumptions of Lemma \ref{lem}, so the $\theta$-weight of a $(P_3,\rho)$-colored broken sheet diagram is a lower bound on the surface-knot's triple point number.

\begin{figure}[h]
\centering
\begin{overpic}[unit=.5mm,scale=.7]{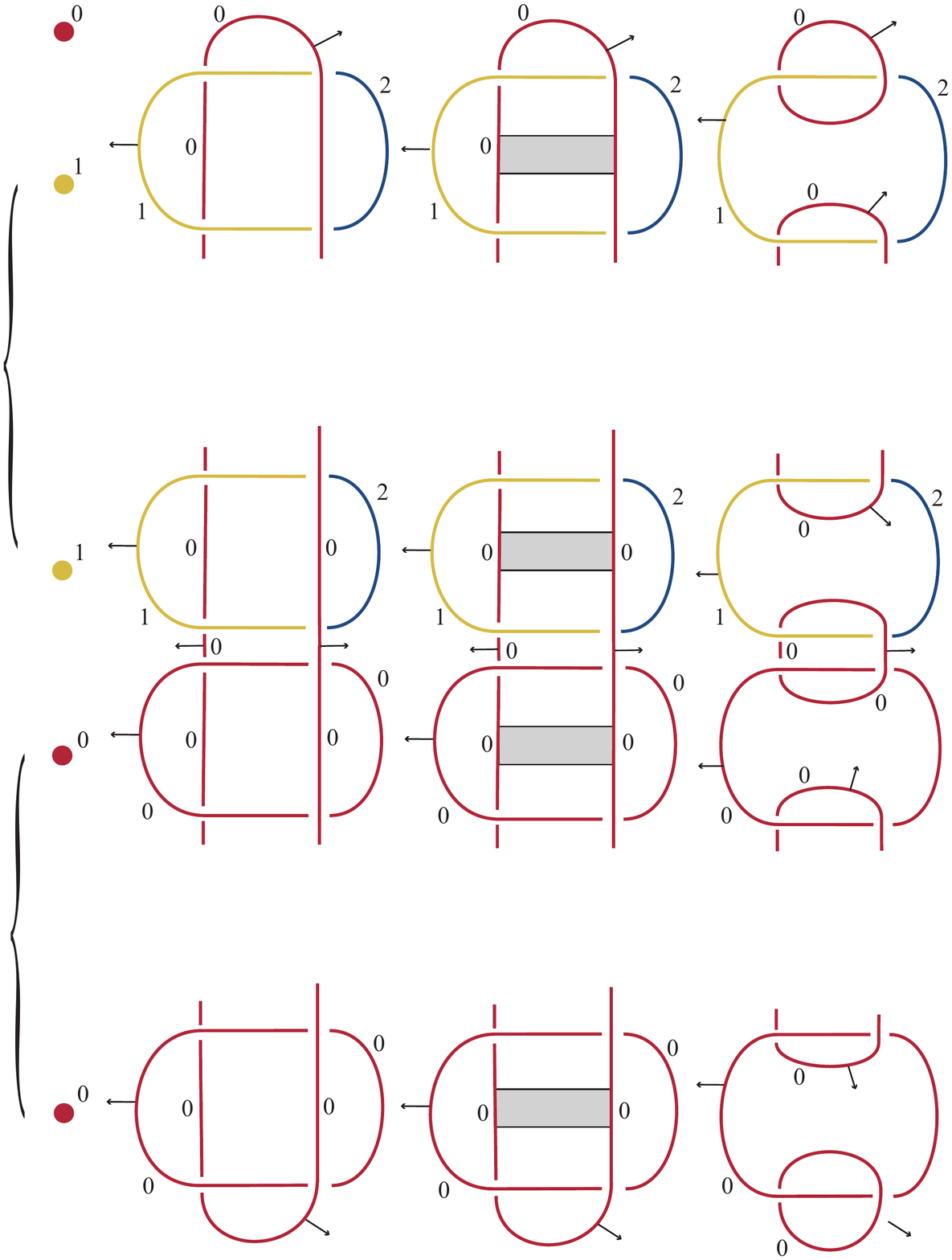}
\put(-8,273){$m$}\put(-7,99){$k$}
\put(80,276){\vdots}
\put(20,276){\vdots}

\put(169,276){\vdots}\put(255,276){\vdots}
\put(169,100){\vdots}\put(255,100){\vdots}\put(79,100){\vdots}\put(20,100){\vdots}
\put(12,-20){(min)}\put(75.5,-20){(i)}\put(164,-20){(ii)}\put(250,-20){(iii)}

\end{overpic}

\vspace{2cm}

\caption{$(P_3,\rho)$-colored motion picture of $F$, 1 of 5.} 
\label{fig:motion1}
\end{figure}

\begin{figure}[h]
\centering

\begin{overpic}[unit=.5mm,scale=.7]{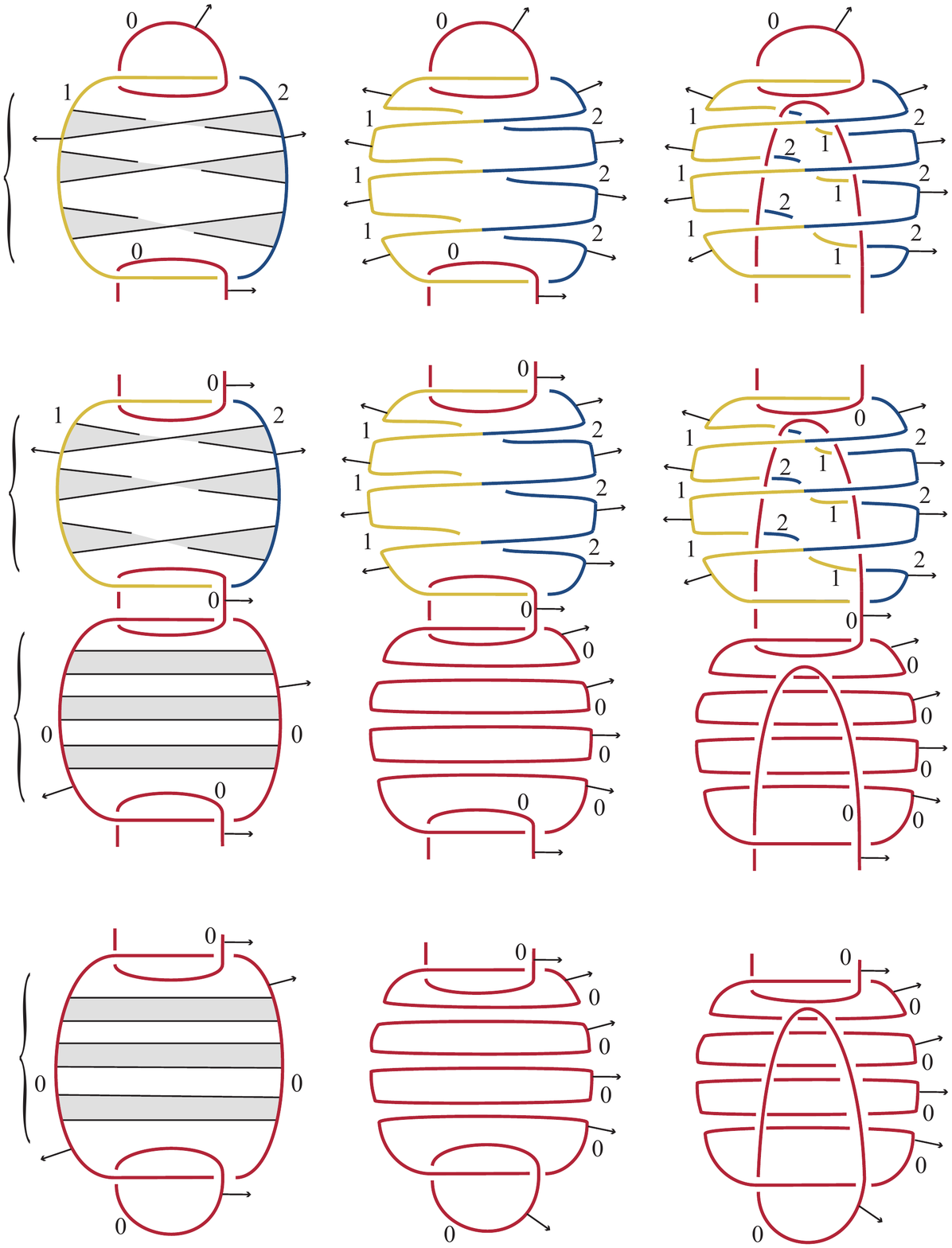}
\put(-7,329){$\frac{g'_1}{2}$} \put(-7,233){$\frac{g'_m}{2}$}\put(-3,166){$g_1$}\put(-3,60){$g_k$}

\put(54,281){\vdots}

\put(150,281){\vdots}

\put(249,281){\vdots}

\put(53,113){\vdots}

\put(149,107){\vdots}

\put(249,104){\vdots}

\put(47,-20){(iv)}\put(146,-20){(v)}\put(244,-20){(vi)}

\end{overpic}

\vspace{2cm}

\caption{$(P_3,\rho)$-colored motion picture of $F$, 2 of 5.}

\label{fig:2}

\end{figure}

\begin{figure}[h]
\centering
\begin{overpic}[unit=.5mm,scale=.7]{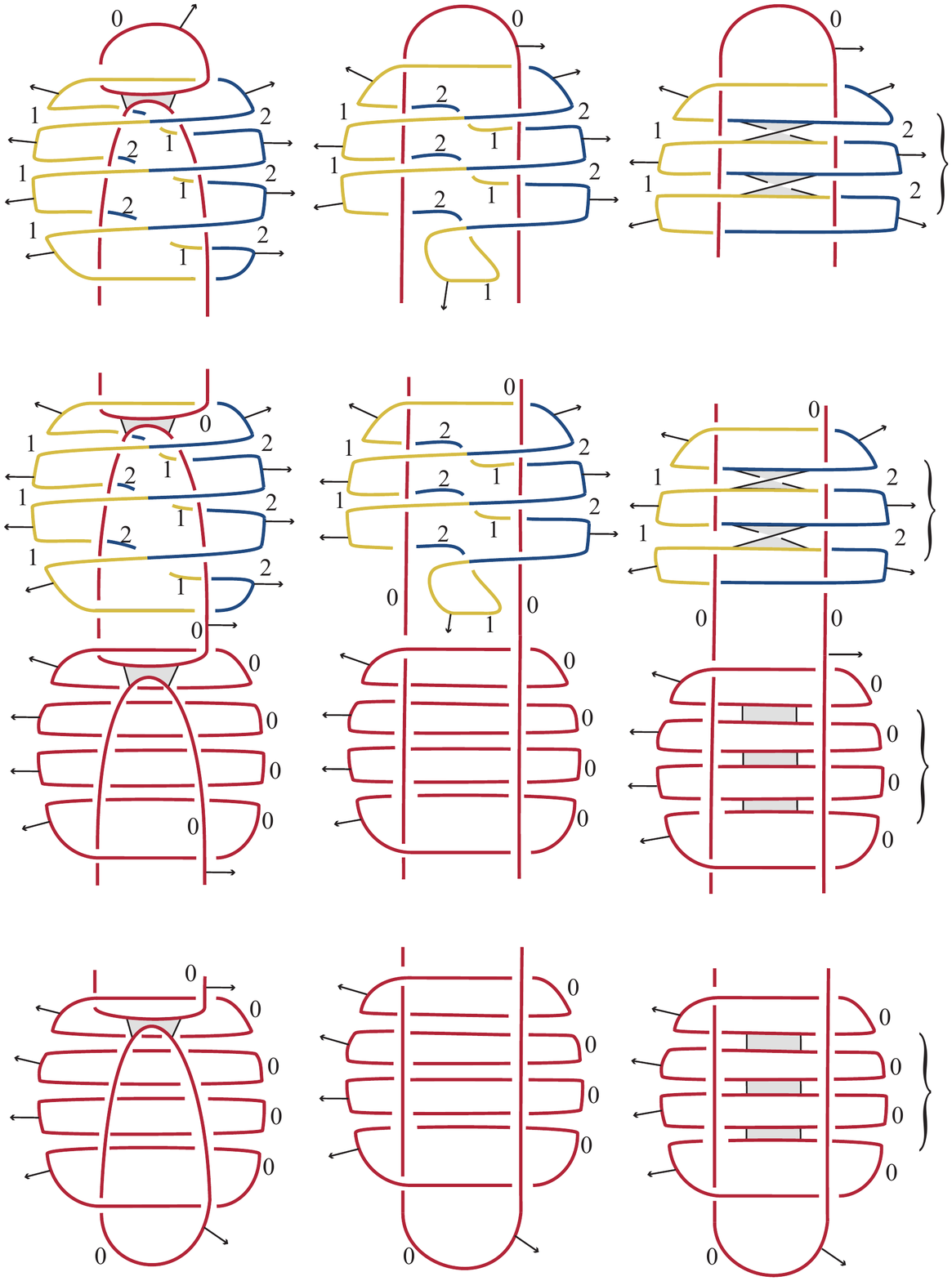}

\put(300,337){$\frac{g'_1}{2}-1$}\put(297,233){$\frac{g'_m}{2}-1$}
\put(295,157){$g_1$}\put(295,59){$g_k$}

\put(50,-20){(vii)}\put(142,-20){(viii)}\put(237,-20){(ix)}

\put(56,282){\vdots}
\put(150,282){\vdots}

\put(244,287){\vdots}

\put(55,105){\vdots}\put(150,109){\vdots}\put(243,105){\vdots}

\end{overpic}

\vspace{2cm}

\caption{$(P_3,\rho)$-colored motion picture of $F$, 3 of 5.}
\label{fig:3}
\end{figure}

\begin{figure}[h]

\centering

\begin{overpic}[unit=.5mm,scale=.7]{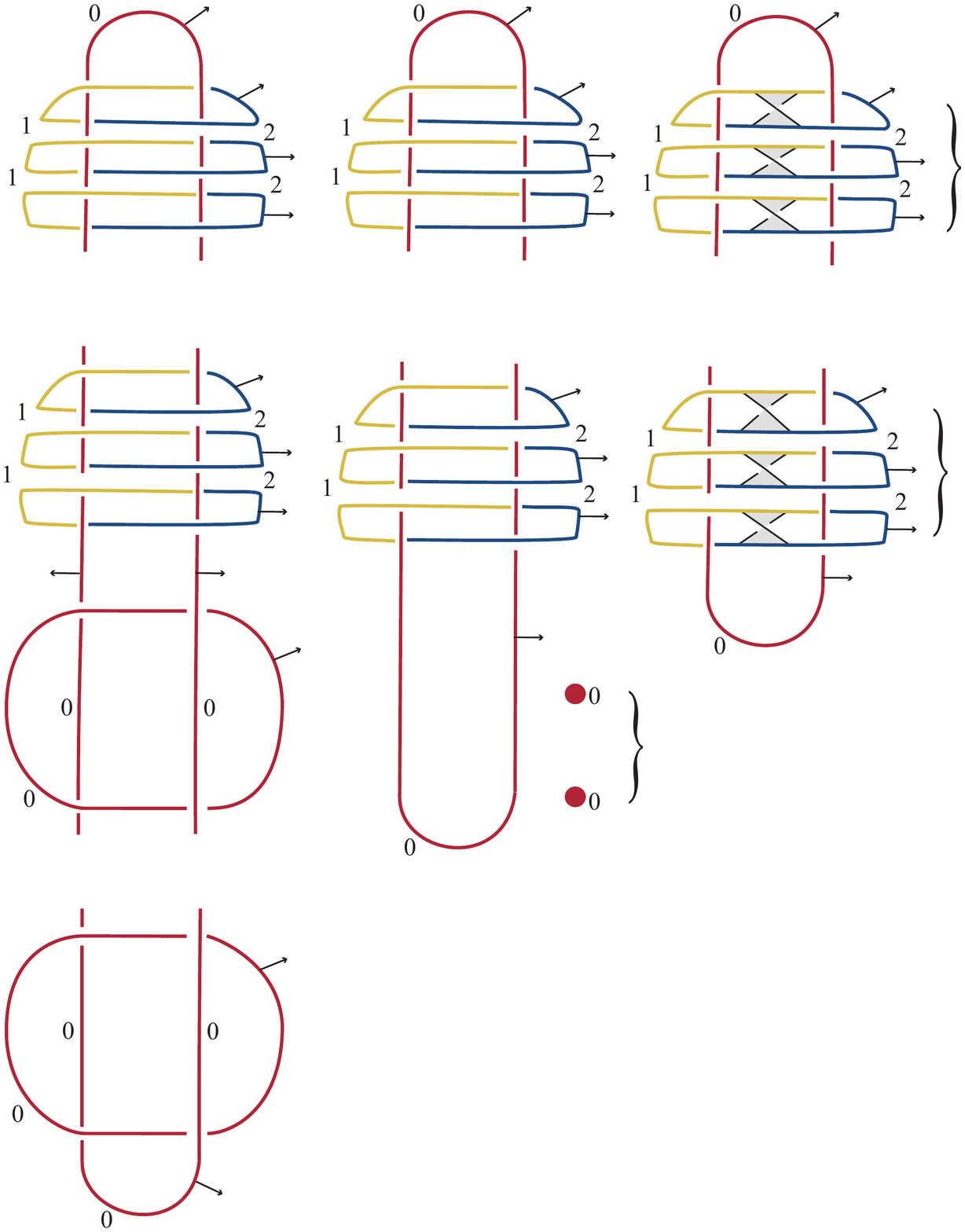}

\put(204,157){$k$}

\put(298,328){$\frac{g'_1}{2}$} \put(296,238){$\frac{g'_m}{2}$}

\put(51,287){\vdots}\put(146,287){\vdots}\put(240,287){\vdots}\put(49,119){\vdots}
\put(179,157){\vdots}

\put(46,-20){(x)}\put(130,-20){(xi) / (max)}\put(230,-20){(xii)}

\end{overpic}

\vspace{2cm}

\caption{$(P_3,\rho)$-colored motion picture of $F$, 4 of 5.}
\label{fig:4}
\end{figure}

\begin{figure}[h]
\centering

\begin{overpic}[unit=.5mm,scale=.7]{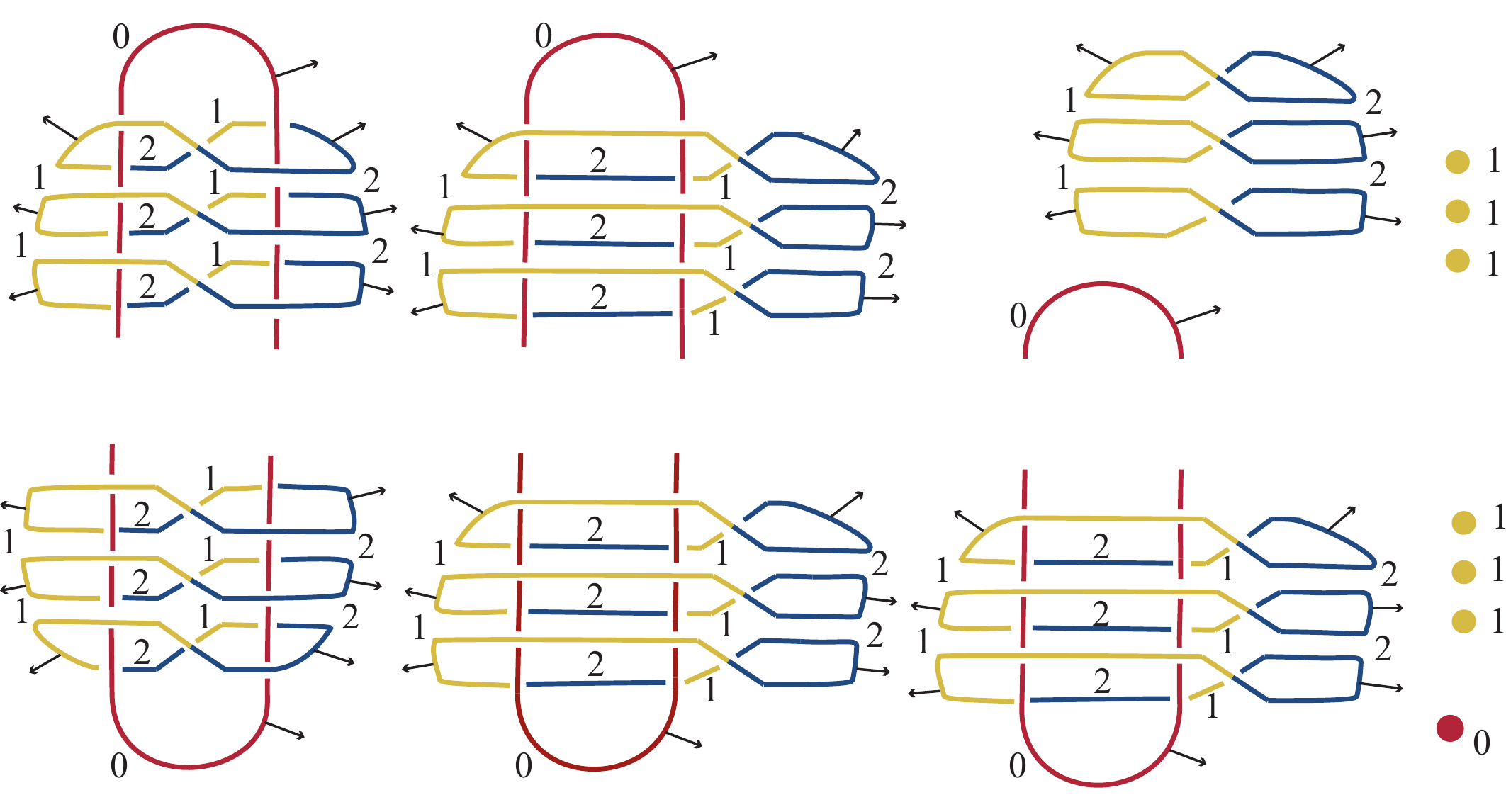}

\put(30,-20){(xiii)}\put(122,-20){(xiv)}\put(222,-20){(xv)}\put(284,-20){(max)}
\put(38,80){\vdots}\put(120,80){\vdots}\put(220,77){\vdots}\put(293,79){\vdots}

\end{overpic}

\vspace{2cm}

\caption{$(P_3,\rho)$-colored motion picture of $F$, 5 of 5.}
\label{fig:5}
\end{figure}

\begin{proof}[Proof of Theorem \ref{thm}.]
\normalfont 
Let $F$ be the surface-link whose motion picture is shown in Figures \ref{fig:motion1}, \ref{fig:2}, \ref{fig:3}, \ref{fig:4}, and \ref{fig:5}. This motion picture is $(P_3,\rho)$-colored. The first still shows the sole minimum of each component. There are $m$ non-orientable components $F'_i$, $k$ orientable components $F_i$, and one component $G$ such that $\cup_{i=1}^m F'_i \ \cup \ \cup_{i=1}^k F_i$ is a trivial surface-link. 

There are $2(k+m)$ saddle points, one minimum, and one maximum in the motion picture restricted to $G$. Therefore, the trivial orientable surface-knot $G$ has genus $k+m$. For any non-negative $g_i$, there are $g_i$ saddle points of $F_i$ in still (iv) and $g_i$ in still (ix). Each $F_i$ has a sole maximum, shown in still (xi). Therefore, $F_i$ has genus $g_i$ and bounds an obvious handlebody. For any positive even $g'_i$, there are $\frac{g'_i}{2}$ saddle points of $F'_i$ in still (iv), $\frac{g'_i}{2}-1$ in still (ix), and $\frac{g'_i}{2}$ in still (xii). Since each $F'_i$ has $\frac{g'_i}{2}$ maxima, $F'_i$ has genus $g'_i$.  

There are $\sum_{i=1}^m \frac{g'_i}{2}$ Reidemeister III moves between still (v) and (vi). These moves induce negative triple points each colored $(2,0,2)$. The sum of their $\theta$-weights is $0\oplus \sum_{i=1}^m \frac{g'_i}{2}$. Between still (xiii) and (xiv), there are $\sum_{i=1}^m \frac{g'_i}{2}$ Reidemeister III moves giving positive triple points each colored $(1,0,2)$. The $\theta$-weight sum of these triple points is $0\oplus \sum_{i=1}^m \frac{g'_i}{2}.$ Therefore, the $\theta$-weight of the induced broken sheet diagram  is $0\oplus \sum_{i=1}^m g'_i$, and Lemma \ref{lem} implies that $t(F = \cup_{i=1}^k F_i \ \cup \ \cup_{i=1}^m F'_i \ \cup \ G)=\sum_{i=1}^m g'_i$ since the induced broken sheet diagram has $\sum_{i=1}^m g'_i$ triple points.

\end{proof}

\section*{Acknowledgements} I would like to thank Kanako Oshiro for her support and Jason Joseph for his mentorship, comments, and edits. I would also like to thank Masahico Saito for his review and suggestions, as well as Scott Carter for his guidance in the field.

        \bibliographystyle{amsplain}
\bibliography{proposal}

\end{document}